\documentclass[reqno]{amsart}

\numberwithin{equation}{section} 
\usepackage{amsmath,amsfonts,amssymb,amsthm,latexsym}
\usepackage{enumerate}
\usepackage{cancel}
\usepackage{stackrel}

\usepackage{cases}
\usepackage[square,comma,numbers,sort&compress]{natbib}

\usepackage[colorlinks=true]{hyperref}
\hypersetup{urlcolor=blue, citecolor=red}

\usepackage[dvipsnames]{xcolor}

\newcounter{mnote}
\theoremstyle{plain}
\newtheorem{theorem}{Theorem}[section]
\newtheorem{proposition}[theorem]{Proposition}

\theoremstyle{definition}

\theoremstyle{remark}
\newtheorem{remark}[theorem]{Remark}

\newcommand{\vect}[1]{\mathbf{#1}}

\newcommand{\bu}{\vect{u}}
\newcommand{\bv}{\vect{v}}

\newcommand{\bx}{\vect{x}}

\newcommand{\field}[1]{\mathbb{#1}}
\newcommand{\nN}{\field{N}}
\newcommand{\nT}{\field{T}}

\newcommand{\nR}{\field{R}}

\newcommand{\sand}{\quad\text{and}\quad}

\newcommand{\pd}[2]{\frac{\partial #1}{\partial #2}}
\newcommand{\od}[2]{\frac{d #1}{d #2}}

\newcommand{\abs}[1]{\left\lvert#1\right\rvert}
\newcommand{\norm}[1]{\left\lVert#1\right\rVert}

\newcommand{\Lph}[1]{\text{$L^{#1}$}(\Om)}
\newcommand{\Hph}[1]{\text{$H^{#1}$}(\Om)}

\newcommand{\Lp}[1]{\text{${L}^{#1}$}(\Om)}
\newcommand{\Hp}[1]{\text{${H}^{#1}$}(\Om)}
\newcommand{\Wp}[2]{\text{${W}^{#1,#2}$}(\Om)}

\newcommand{\Lpd}[1]{\text{${ L}^{#1}$}(\Om)}

\newcommand{\dU}{\tilde{u}}
\newcommand{\dV}{\tilde{v}}
\newcommand{\dP}{\tilde{p}}
\newcommand{\uu}{u^\ast}
\newcommand{\vv}{v^\ast}
\newcommand{\p}{p^\ast}

\def\aa{\alpha}

\def\la{\lambda}
\def\Dd{\Delta}

\def\Om{\Omega}

\def\pp{\partial}

\newcommand*{\cS}{{\mathcal S}}

\newcommand*{\ttwo}{{\theta_2}}

\begin{document}
\title[abridged data assimilation for 3D $\alpha$ models]{A Data Assimilation Algorithm: The Paradigm of the 3D Leray-$\alpha$ Model of Turbulence.}

\date{February 6, 2017}

%
\author{Aseel Farhat}
\address[Aseel Farhat]{Department of Mathematics\\
               University of Virginia\\
       Charlottesville, VA 22904, USA.}
\email[Aseel Farhat]{af7py@virginia.edu}
\author{Evelyn Lunasin}
\address[Evelyn Lunasin]{Department of Mathematics\\
                United States Naval Academy\\
        Annapolis, MD, 21402 USA.}
\email[Evelyn Lunasin]{lunasin@usna.edu}

\author{Edriss S. Titi}
\address[Edriss S. Titi]{Department of Mathematics, Texas A\&M University, 3368 TAMU,
 College Station, TX 77843-3368, USA.  {Also},
 The Science Program, Texas A\&M University at Qatar, Doha, Qatar.}
   \email{titi@math.tamu.edu}

\begin{abstract}
In this paper we survey the various implementations of a new data assimilation (downscaling) algorithm based on spatial coarse mesh measurements.  As a paradigm, we demonstrate the application of this algorithm to the 3D Leray-$\alpha$ subgrid scale turbulence model. Most importantly, we use  this paradigm to show that it is not always necessary that one has to  collect coarse mesh measurements of all the state variables, that are involved in the underlying evolutionary system, in order to recover the corresponding exact reference solution. Specifically, we show that in the case of the  3D Leray-$\alpha$ model of turbulence  the  solutions of the algorithm, constructed using only coarse mesh observations of {\it any two components of the three-dimensional velocity field}, and without any information of the third component,  converge, at an exponential rate in time, to the corresponding exact reference solution of the 3D Leray-$\alpha$ model. This study serves as an addendum to our recent work on abridged continuous data assimilation for the 2D Navier-Stokes equations. Notably, similar results have also been recently  established for the 3D viscous Planetary Geostrophic circulation model in which we show that coarse mesh measurements of the temperature alone are sufficient for recovering, through our data assimilation algorithm, the full solution; viz.\ the three components of velocity vector field and the temperature.  Consequently, this proves the Charney conjecture for the 3D Planetary Geostrophic model; namely, that the history of  the large  spatial scales of temperature  is sufficient for  determining all the other quantities (state variables) of the model.

\end{abstract}

 \maketitle

  \centerline{\it This paper is dedicated to the memory of Professor Abbas Bahri}

  \vspace{1cm}



\bigskip
\section{Introduction}\label{intro }

Data assimilation is a methodology to estimate weather or ocean variables combining (synchronizing) information from observational data with a numerical dynamical (forecast) model. In the context of control engineering, tracing back since the early 70's, data assimilation was also applied to simpler models in \cite{Leunberger1971, Thau1973, Nijmeijer2001}. One of the classical methods of continuous data assimilation, see, e.g., \cite{Daley, Ghil-Halem-Atlas1978}, is to insert observational measurements directly into a model as the latter is being integrated in time. For example, one can insert Fourier low mode observables into the evolution equation for the high modes, then the values for the low modes and high modes are combined to form a complete approximation of the state of the system. This resulting state value is then used as an initial condition to evolve the forecast model using high resolution simulation. For the 2D Navier-Stokes, this approach was considered in \cite{Browning_H_K, Henshaw_Kreiss_Ystrom,Hayden_Olson_Titi,Olson_Titi_2003, Olson_Titi_2008, Korn,B_L_Stuart, Law-etal2014}. The problem when some state variables observations are not available as an input to the numerical forecast model was studied in \cite{Charney1969, Errico85, Ghil1977, Ghil-Halem-Atlas1978, Hoke-Anthes1976, LorencTibaldi80} for simplified numerical forecast models.

Recent  studies in \cite{FJT, FLT2015, FLT_horizontal, FLT_PG, FLT_porous} have established rigorous analytical support pertaining to a continuous data assimilation algorithm similar to that introduced in \cite{Azouani_Olson_Titi}, which is a feedback control algorithm applied to data assimilation (see also \cite{Azouani_Titi}), previously known as nudging or newtonian relaxation.  In  these cases, the authors have analyzed a data assimilation algorithm for different  systems when some of the state variable observations are not supplied in the algorithm.  In other words,  a rigorous analytical support to a data assimilation algorithm that can identify the full state of the system knowing only coarse spatial mesh observational measurements not of full state variables of the model, but only of the selected state variables in the system, were analytically justified.   In this article we summarize our recent results to provide support of the applicability of the scheme in several model equations.  We then demonstrate the application of this algorithm to the  Leray-$\alpha$ subgrid scale turbulence model which serves as an addendum to our recent work on abridged continuous data assimilation for the 2D Navier-Stokes equations.

Starting from the work of \cite{Azouani_Olson_Titi}, the recent works \cite{FJT, FLT2015, FLT_porous, FLT_horizontal,FLT_PG}, mentioned above provided some stepping stones to rigorous justification to some of the earlier conjectures of meteorologists in numerical weather prediction.   For example, the systematic theoretical framework of the proposed global control scheme allowed the authors in \cite{Azouani_Olson_Titi} and \cite{ FJT} to provide sufficient conditions on the spatial resolution of the collected spatial coarse mesh data and the relaxation parameter that guarantees that the approximating solution obtained from this algorithm converges to the corresponding unknown reference solution over time (with the assumption that the observational data measurements are free of noise).  Without access to concrete theoretical analysis, earlier implementation of the ``nudging" algorithm  left geophysicists searching for the optimal or suitable nudging coefficient (or relaxation parameter) through expensive numerical experiments.  Naturally, one wishes for the availability of a sharper estimate  for the operational characteristic parameters  than what the theoretical results gives. However, these recent analytical results, although not sharp, may allow one to track the correct parameter ranges without the expensive numerical experiments. Computational studies implementing these algorithms under drastically more relaxed conditions were demonstrated \cite{Gesho_Olson_Titi, Altaf} for example.

To understand the value of the development of these analytical results stemming from the series of studies, we should mention its valuable impact in meteorology. In weather prediction, we've mentioned earlier the need to analyze the success of a data assimilation algorithm when some state variable observations are not available. Charney's question in \cite{Charney1969, Ghil-Halem-Atlas1978, Ghil1977} of whether temperature observations are enough to determine all the dynamical state of the system  is an important problem in meteorology and engineering. In \cite{Ghil1977}, an analytical argument suggested that the Charney's conjecture is correct, in particular, for the shallow water model. The authors in \cite{Ghil1977} derived a diagnostic system for the velocity field that gives the velocity in terms of first and second time derivatives of the mass field (the geopotential of the free surface of the fluid). A mathematical argument was then presented to justify that the mass field and its first and second time derivatives determine the velocity field fully. Similar diagnostic systems can be derived for other simple primitive equation models. The work in \cite{Ghil1977} gave a precise theoretical formulation of the Charney conjecture for certain simple atmospheric models.

Numerical tests in \cite{Ghil1977} suggested that in practice, it can be hard to implement this method to solve for the velocity field using only measurements on the mass field. Further numerical testing in \cite{Ghil-Halem-Atlas1978} affirmed that it is not certain whether assimilation with temperature data alone will yield initial states of arbitrary accuracy.
The authors in \cite{Ghil-Halem-Atlas1978} considered the primitive equations (the main weather forecast model) and tested and compared different time-continuous data assimilation methods using temperature data alone. In their numerical experiments, they concluded that the accuracy of the assimilation is highly dependent on  assimilation method used and  on the integrity of the measured observational temperature data.  A relevant recent numerical study on a data assimilation algorithm for the 2D B\'enard system \cite{Altaf}, inspired by the work in \cite{FJT}, showed that it is sufficient to use coarse velocity measurements in the algorithm to recover the full true state of the system. On the other hand, it was concluded in \cite{Altaf} that assimilations using coarse temperature measurements only will not always recover the true state of the system. It was observed that the convergence to the true state using temperature measurements only is actually sensitive to the amount of noise in the measured data as well as to the spacing (the sparsity of the collected data) and the time-frequency of such measured temperature data. These results in \cite{Altaf} are consistent with the results of the earlier numerical experiments  in \cite{Ghil1977} and \cite{Ghil-Halem-Atlas1978}.

These results may indicate that the answer to Charney's question is negative for practical issues we have with our measuring equipments or our numerical solving techniques. More recently, in \cite{FLT_PG}, we proposed an improved continuous data assimilation algorithm for the 3D Planetary Geostrophic model that requires observations of the temperature only. We provided a rigorous mathematical argument that temperature history of the atmosphere will determine other state variables for this specific planetary scale model, thus justifying theoretically Charney's conjecture for this model. Numerical implementation of our algorithm for the 3D Planetary Geostrophic model is subject to a new work to compare with the numerical results obtained in \cite{Altaf},  \cite{Ghil-Halem-Atlas1978} and \cite{Ghil1977}.

We will review relevant results starting from the algorithm introduced in \cite{Azouani_Olson_Titi}. The algorithm in \cite{Azouani_Olson_Titi} can be formally described as follows: suppose that $u(t)$ represents a solution of some dynamical system governed by an evolution equation of the type
\begin{align}\label{dissipative}
\od{u}{t} = F(u),
\end{align}
where the initial data $u(0)= u_{in}$ is missing. Let $I_h(u(t))$ represent an interpolant operator based on the observations of the  system at a coarse spatial resolution of size $h$, for $t\in [ 0,T ]$.
The algorithm proposed in \cite{Azouani_Olson_Titi} is to construct a solution $v(t)$ from the observations that satisfies the equations
\begin{subequations}\label{du}
\begin{align}
&\od{v}{t} = F(v) - \mu (I_h(v)- I_h(u)), \\
&v(0)= v_{in},
\end{align}
\end{subequations}
where $\mu>0$ is a relaxation (nudging) parameter and $v_{in}$ is taken to be arbitrary initial data.
As mentioned from the previous literature, this particular algorithm was designed to work for general  dissipative dynamical systems of the form \eqref{dissipative} that are known to have global, in time, solutions, finite-dimensional global attractor and a finite set of determining parameters (see, e.g., \cite{C_O_T, F_M_R_T, Foias_Prodi, Foias_Temam, Foias_Temam_2, Foias_Titi, Holst_Titi, Jones_Titi, Jones_Titi_2} and references therein). Lower bounds on $\mu>0$ and upper bounds on $h>0$ can be derived such that the approximate solution $v(t)$ is converging in time to the reference solution $u(t)$. These estimates are not sharp (see the numerical results in \cite{Altaf} and \cite{Gesho_Olson_Titi}) since their derivation uses the existing estimates for the global solution in the global attractor of these dissipative systems.

In the context of the incompressible 2D Navier-Stokes equations (NSE), the authors in \cite{Azouani_Olson_Titi} studied the conditions under which the approximate solution $v(t)$ obtained by the algorithm in \eqref{du} converges to the reference solution $u(t)$ over time (see also \cite{Gesho_Olson_Titi}).  In \cite{ALT2014}, it was shown that approximate solutions constructed using observations on all three components of the unfiltered velocity field converge in time to the reference solution of the 3D Navier-Stokes-$\alpha$ model. Another application of this algorithm i for the three-dimensional Brinkman-Forchheimer-Extended Darcy model was introduced in \cite{MTT2015}. The authors in \cite{J_M_T} studied the convergence of the algorithm to the reference solution in the case of the two-dimensional subcritical surface quasi-geostrophic (SQG) equation. The convergence of this synchronization algorithm for the 2D NSE, in higher order (Gevery class) norm and in $L^\infty$ norm, was later studied in \cite{B_M}. An extension of the approach in \cite{Azouani_Olson_Titi} to the case when the observational data contains stochastic noise was analyzed in \cite{Bessaih-Olson-Titi}. A study of the the algorithm for the 2D NSE when the measurements are obtained discretely in time and are contaminated by systematic error is presented in \cite{FMTi}. More recently in \cite{M_T}, the authors obtain uniform in time estimates for the error between the numerical approximation given by the Post-Processing Galerkin method of the downscaling algorithm and the reference solution, for the 2D NSE.   Notably, this uniform in time error estimates provide a strong evidence for the practical reliable numerical implementation of this algorithm.

In \cite{FJT},  a continuous data assimilation scheme for the two-dimensional incompressible B\'enard convection problem was introduced. The data assimilation algorithm in \cite{FJT} constructed the approximate solutions for the velocity $u$ and temperature fluctuations using only the observational data of the velocity field and without any measurements for the temperature fluctuations. In \cite{FLT2015}, we introduced an abridged dynamic continuous data assimilation for the 2D NSE inspired by the recent algorithms introduced in \cite{Azouani_Olson_Titi, FJT}. There we establish convergence results for the improved algorithm where the observational data needed to be measured and inserted into the model equation is reduced or subsampled.   Our algorithm required observational measurements of only one component of the velocity vector field.  The underlying analysis were made feasible by taking advantage of the divergence free condition on the velocity field. Our work in \cite{FLT2015}, was then applied and extended for the  convergence analysis for a 2D B\'enard convection problem, where the approximate solutions  constructed using observations in only the horizontal component of the two-dimensional velocity field and without any measurements on the temperature, converge in time to the reference solution of the 2D B\'enard system. This was a progression of the recent result in \cite{FJT} where convergence results were established, given that observations are known on all of the components of the velocity field and without any measurements of the temperature. In \cite{FLT_porous} we propose that a data assimilation algorithm based on temperature measurements alone can be designed for the B\'enard convection in porous medium.  In this work it was established that requiring sufficient amount of coarse spatial observational measurements of only the temperature measurements  as input is able to recover the full state of the system. Subsequently, in \cite{FLT_PG} we proposed an improved data assimilation  algorithm for recovering the exact full reference solution (velocity and temperature field) of the 3D Planetary Geostrophic model, at an exponential rate in time, by employing  coarse spatial mesh observations of the temperature alone. In particular, we presented a rigorous justification to an earlier conjecture of Charney which states that temperature history of the atmosphere, for certain simple atmospheric models, determines all other state variables.

\bigskip

\section{Application to Turbulence models}\label{alpha-model}
All the analysis of the proposed data assimilation algorithm assumes the global existence of the underlying model and uses previously known estimates.  It is for this reason that we are not able to prove similar results for the 3D NSE case, even though numerical testing may be applicable and feasible.  Note, however, that we are able to formulate the analytical setting for a family of globally well-posed subgrid scale turbulence models belonging to a family called $\alpha-$models of turbulence.  These are simplified models through an averaging process that is designed to capture the large scale dynamics of the flow and at the same time provide reliable closure model to the averaged equations.  The first member of the family is introduced in the late 1990's called the Navier-Stokes-$\alpha$ (NS-$\alpha$) model (also known as Lagrangian averaged Navier-Stokes-$\alpha$ (LANS-$\alpha$)  or viscous Camassa-Holm equations \cite{CH98, CH99, CH00, FHTP, FHTM}) is written as follows:
\begin{subequations}
\begin{align}
\partial_t v + u\cdot\nabla v + \nabla u \cdot v + \nabla p = \nu\Delta v + f,\\
\nabla\cdot u=0, \mbox{ and } v = u-\alpha^2\Delta u.
\end{align}
\end{subequations}
Unlike other subgrid closure models which normally add some additional dissipative process, this new modeling approach regularizes the NSE by restructuring the distribution of the energy in the wave number $k > 1/\alpha$ of the inertial range \cite{FHTP}. In other words, NS-$\alpha$ smooths the nonlinearity of the NSE, instead of enhancing dissipation.   Many other $\alpha$ models, such as the Leray-$\alpha$ \cite{CHOT}, the Clark-$\alpha$ \cite{CHTi}, the Navier-Stokes-Voigt (NSV) equation \cite{KLT07, KT07, Osk73, Osk80}, and the models we have introduced, namely, the modified Leray-$\alpha$ (ML-$\alpha$) \cite{ILT}, the simplified Bardina model (SBM) \cite{CLT, LL2006}, and the NS-$\alpha$-like models \cite{Olson_Titi_2007},  were inspired by this regularization technique.  These models can be represented by a generalized model of the form
\begin{subequations}\label{e:pde}
\begin{align}
\partial_t u + Au + (M u\cdot \nabla) (N u) + {\chi} \nabla (M u)^T\cdot (N u) 
  + \nabla p &= f(x),\\
\nabla\cdot u &= 0,\\
u(0,x) & = u_{in}(x),
\end{align}
\end{subequations}
where
$A$, $M$, and $N$ are bounded linear operators having certain mapping properties,
${\chi}$ is either 1 or 0,
${\theta}$  controls the strength of the dissipation operator $A$,
and the two parameters which control the degree of smoothing in the operators $M$ and $N$, respectively, are ${\theta_1}$ and ${\theta_2}$. Table 1 summarizes certain $\alpha-$models of turbulence.

\begin{table}[h!]
\caption{\footnotesize
Some special cases of the model \eqref{e:pde} with $\alpha>0$,
and with $\mathcal{S}=(I-\alpha^2\Delta)^{-1}$
and $\mathcal{S}_{\theta_2}~=~[I~+~(-\alpha^2\Delta)^{\theta_2}]^{-1}$.
}
\begin{center}
\begin{tabular}{cccccccc}
\hline
Model  &  {\bf NSE}      & {\bf Leray-$\alpha$} & {\bf ML-$\alpha$} & {\bf SBM}        & {\bf NSV}               & {\bf NS-$\alpha$}       & {\bf NS-$\alpha$-like} \\ \hline\hline
 $A$   & $-\nu\Dd$ & $-\nu\Dd$      & $-\nu\Dd$   & $-\nu\Dd$  & $-\nu\Dd\cS$ & $-\nu\Dd$         & $\nu (-\Dd)^{\theta}$\\
 $M$   & $I$       & $\cS$     & $I$         & $\cS$ & $\cS$        & $\cS$        & $\cS_\ttwo$\\
 $N$   & $I$       & $I$            & $\cS$  & $\cS$ & $\cS$        & $I$               & $I$\\
$\chi$ & 0         &0               &0            & 0          &0                  &1                  &1\\
\hline
\end{tabular}
\end{center}
\label{t:spec}
\end{table}

All of the models just mentioned have global regular solutions and posses fewer degrees of freedom than the NSE.  Moreover, mathematical analysis also proved that the solutions to these models converge to the solution of NSE in the limit as the filter width parameter $\alpha$ tends to zero.  In addition, several of the $\alpha$-models of turbulence have been tested against averaged empirical data collected from turbulent channels and pipes, for a wide range of Reynolds numbers (up to $17 \times 10^6$) \cite{CH98, CH99, CH00}.  The successful analytical, empirical and computational aspects (see for example \cite{FHTP, HHPW08, LKTT1, LKTT2} and references therein) of  the alpha turbulence models have attracted numerous applications, see for example \cite{BLT07} for application to the quasi-geostrophic equations, \cite{KhTi07} for application to Birkhoff-Rott approximation dynamics of vortex sheets of the 2D Euler equations, \cite{LinTi07, MMP05a, MMP05b} for applications to incompressible magnetohydrodynamics equations.  See also \cite{LLT12a, LLT12b} for the $\alpha$-regularization of the inviscid 3D Boussinesq equation. A unified analysis of an additional family of $\alpha$-type regularized model, also called a  general family of regularized Navier-Stokes and MHD models on $n$-dimensional smooth compact Riemannian manifolds with or without boundary, with $n \geq 2$ is studied in \cite{HLT10}.  For  approximate deconvolution models of turbulence see \cite{ Layton_Neda2007, Layton_Rebholz2012}. For other closure models see \cite{BI_Layton} and references therein.

The proposed algorithms in \cite{FJT, FLT2015, FLT_horizontal, FLT_PG, FLT_porous} sparked an idea that perhaps for this $\alpha$-model one can construct approximate solutions using only observations in the horizontal components and without any measurements on the vertical component of the velocity field converge in time to the reference solution. This is indeed the case, made possible by taking advantage of the divergence free condition for the velocity field. Similar results can be claimed for the other certain $\alpha$-models. In this paper, we apply the data assimilation algorithm for the case of  3D Leray-$\alpha$ model which we recall below \cite{CHOT}:
\begin{subequations}\label{Leray-alpha}
\begin{align}
\pp_t v  -\nu \Delta v+ (u\cdot\nabla) v &= -\nabla p + f,\\
\nabla \cdot u &=  \nabla \cdot v = 0,\\
v &= u-\aa^2\Dd u,\\
v(0, x,y,z) &= v_{in}(x,y,z),\\
u , v \mbox{ and } p  \mbox{ are periodic, with basic periodic box } \Omega &= [0, L]^3.
\end{align}
\end{subequations}
The nonlinearity is advected by the smoother velocity field and notice that consistent with all the other alpha models, the above system is the Navier-Stokes system of equations when $\aa = 0$, i.e. $u = v$.

Let $I_h(\varphi)$ represent an interpolant operator based on the observational measurements of the scalar function $\varphi$ at a coarse spatial resolution of size $h$.
Given the viscosity $\nu$, the proposed algorithm for reconstructing $u(t)$ and $v(t)$ from only the horizontal observational measurements, which are represented by means of the interpolant operators  $I_h(v_1(t))$ and  $I_h(v_2(t))$ for $t\in [0, T]$ is given by the system
\begin{subequations}\label{Leray-alpha-DA}
\begin{align}
\pp_t \vv_1  -\nu \Delta \vv_1+ (\uu\cdot\nabla) \vv_1 &= -\pp_x \p + \mu\left(I_h(v_1) - I_h(\vv_1)\right)+f_1,\\
\pp_t \vv_2  -\nu \Delta \vv_2+ (\uu\cdot\nabla) \vv_2 &= -\pp_y \p + \mu\left(I_h(v_2) - I_h(\vv_2)\right)+f_2,\\
\pp_t \vv_3  -\nu \Delta \vv_3+ (\uu\cdot\nabla) \vv_3 &= -\pp_x \p +f_3,\\
\nabla \cdot \uu &=  \nabla \cdot \vv = 0,\\
\vv &= \uu-\aa^2\Dd \uu,\\
\vv(0, x,y,z) &= \vv_{in}(x,y,z).
\end{align}
\end{subequations}
supplemented with periodic boundary conditions, where $\mu$ is again a positive parameter which relaxes (nudges) the coarse spatial scales of $v$ toward those of the observed data.

Consequently, $\vv(t,x,y,z)$ is the approximating velocity field, with $\vv(0,x,y,z) = \vv_{in}(x,y,z)$ taken to be arbitrary. We note that any data assimilation algorithm using two out of three components of the velocity field also works. Here, observational data of the horizontal components $I_h(v_1(t))$ and $I_h(v_2(t))$ were chosen as an example.

We will consider interpolant observables given by linear interpolant operators $I_h: \Hph{1} \rightarrow \Lph{2}$, that approximate identity and satisfy the approximation property
\begin{align}\label{app}
\norm{\varphi - I_h(\varphi)}_{\Lph{2}} \leq \gamma_0h\norm{\varphi}_{\Hph{1}}, 
\end{align}
for every $\varphi$ in the Sobolev space $\Hph{1}$.  One example of an interpolant observable of this type is the orthogonal projection onto the low Fourier modes with wave numbers $k$ such that $|k|\leq 1/h$. A more physical example are the volume elements that were studied in \cite{Jones_Titi}. A second type of linear interpolant operators $I_h: \Hph{2}\rightarrow\Lph{2}$, that satisfy the approximation property
\begin{align}\label{app2}
\norm{\varphi - I_h(\varphi)}_{\Lph{2}} \leq \gamma_1h\norm{\varphi}_{\Hph{1}} + \gamma_2h^2\norm{\varphi}_{\Hph{2}},
\end{align}
for every $\varphi$ in the Sobolev space $\Hph{2}$, can be considered with this algorithm. An example of this type of interpolant observables is given by the measurements at a discrete set of nodal points in $\Omega$ (see Appendix A in \cite{Azouani_Olson_Titi}). The treatment for the second type of interpolant is slightly more technical (see, e.g. \cite{FJT, FLT2015}), and thus we won't consider it here for the sake of keeping this note concise.

We prove an analytic upper bound on the spatial resolution $h$ of the observational measurements and an analytic lower bounds on the relaxation parameter $\mu$ that is needed in order for the proposed algorithm \eqref{Leray-alpha-DA} to recover the reference solution of the 3D Leray-$\alpha$ system \eqref{Leray-alpha} that corresponds to the coarse measurements. These bounds depend on physical parameters of the system, the Grashof number as an example. We remark that extensions of algorithm \eqref{Leray-alpha-DA}, for the cases of measurements with stochastic noise and of discrete spatio-temporal measurements with systematic error, can be established by combining the ideas we present here with the techniques reported in \cite{Bessaih-Olson-Titi} and \cite{FMTi}, respectively.

\bigskip

\section{Preliminaries}\label{pre}

We define $\mathcal{F}$ to be the set of divergence free $L$-periodic trigonometric polynomial vector fields from $\nR^3 \rightarrow \nR^3$, with spatial average zero over $\Omega$. We denote by $\Lp{2}$, $\Wp{s}{p}$, and $\Hp{s}\equiv \Wp{s}{2}$ the usual Sobolev spaces in three-dimensions, and we denote by $H$ and $V $ the closure of $\mathcal{F}$ in $\Lp{2}$ and $\Hp{1}$, respectively.

We denote the dual of $V $ by $V ^{'}$ and the Helmholz-Leray projector from $\Lp{2}$ onto $H$ by $P_\sigma$.
The Stokes operator $A:V  \rightarrow V ^{'}$, can now be expressed as
$$ Au = -P_\sigma \Delta u, $$
for each $u, v \in V $.  We observe that in the periodic boundary condition case $A~= ~-\Delta.$ The linear operator $A$ is self-adjoint and positive definite with compact inverse $A^{-1}: H \rightarrow H$. Thus, there exists a complete orthonormal set of eigenfunctions $w_i$ in $H$ such that $Aw_i= \lambda_iw_i$ where $0<\lambda_i\leq\lambda_{i+1}$ for $i\in \nN$.  The domain of $A$ will
be written as $\mathcal{D}(A)= \{u\in V : Au \in H\}$.

We define the inner products on $H$ and $V$ respectively by
\[(\bu,\bv)=\sum_{i=1}^3\int_{\nT^3} u_iv_i\,d\bx
\sand
((\bu,\bv))=\sum_{i,j=1}^3\int_{\nT^3}\partial_ju_i\partial_jv_i\,d\bx,
\]
and their associated norms $(\bu,\bu)^{1/2} =\norm{u}_{\Lp{2}}$, $((\bu,\bu))^{1/2}=\norm{A^{1/2}\bu}_{\Lp{2}}$.
Note that $((\cdot,\cdot))$ is a norm due to the Poincar\'e inequality
\begin{equation}\label{poincare}
 \|\phi\|_{\Lp{2}}^2\leq \lambda_1^{-1}\|\nabla\phi\|_{\Lp{2}}^2, \quad\mbox{ for all }  \phi \in V,
\end{equation}
where $\lambda_1$ is the smallest eigenvalue of the operator $A$ in  three-dimensions,  subject to periodic boundary conditions.

We use the the following inner products in $\Hp{1}$ and $\Hp{2}$, respectively 
\[((u,v))_{\Hp{1}} = \lambda_1 \left[(u,v) + \alpha^2(A^{1/2} u, A^{1/2} v)\right], \]
and
\[((u,v))_{\Hp{2}} = \lambda_1^2\left[(u,v) + 2\alpha^2(A^{1/2} u, A^{1/2} v) + \alpha^4(A u, A v)\right].
\]
The above inner products were used in \cite{CHOT} so that the norms in $\Hp{1}$ and $\Hp{2}$ are dimensionally homogeneous. Using these definitions, one can observe that
\begin{align}\label{chot}
\lambda_1 \norm{v}_{\Lp{2}} \leq \norm{u}_{\Hp{2}}\leq 2 \lambda_1 \norm{v}_{\Lp{2}},
\end{align}
where $v = u - \alpha^2\Delta u$.

\begin{remark}
We will use these notations indiscriminately for both scalars and vectors, which should not be a source of confusion.
\end{remark}

Let $Y$ be a Banach space.  We denote by $L^p([0,T];Y)$ the space of (Bochner) measurable functions $t\mapsto w(t)$, where $w(t)\in Y$, for a.e. $t\in[0,T]$, such that the integral $\int_0^T\|w(t)\|_Y^p\,dt$ is finite.

Hereafter, $c$ denotes a universal dimensionless positive constant. Our estimates for the nonlinear terms will involve the Sobolev inequality in three-dimensions:
\begin{align}\label{sobolev}
\norm{u}_{\Lp{\infty}} \leq c \lambda_1^{-1/4} \norm{u}_{\Hp{2}}.
\end{align}
Furthermore, inequality \eqref{app} implies that
\begin{align}\label{app_F}
\norm{w-I_h(w)}_{\Lp{2}}^2\leq c_0^2h^2\norm{A^{1/2}w}_{\Lp{2}}^2,
\end{align}
for every $w\in V$, where $c_0=\gamma_0$.

Here, $G$ denotes the the Grashof number in three-dimensions
\begin{align}\label{Grashof_3}
G = \frac{\norm{f}_{\Lp{2}}}{\nu^2\lambda_1^{3/4}}.
\end{align}

We recall that the 3D Leray-$\alpha$ model \eqref{Leray-alpha}, subject to periodic boundary conditions,  is well-posed and posses a finite-dimensional compact global attractor.
\begin{theorem}[Existence and uniqueness]\cite{CHOT}
If $v_{in} \in V$ and $f\in H$, then, for any $T>0$, the 3D Leray-$\alpha$ model \eqref{Leray-alpha} has a unique global strong solution $v(t, x,y,z)= (v_1(t, x,y,z), v_2(t , x,y,z), v_3(t , x,y, z))$ that satisfies
\begin{align*}
v \in C([0,T];V ) \cap L^2([0,T];\mathcal{D}(A)), \qquad \text{and} \qquad \od{v}{t} \in L^2([0,T];H).
\end{align*}
Moreover, the system admits a finite dimensional global attractor $\mathcal{A}$ that is compact in $H$.
\end{theorem}

The following bounds on solutions $v$ of \eqref{Leray-alpha} can be proved using the estimates obtained in \cite{CHOT}.
\begin{proposition} \cite{CHOT}
Let $\tau>0$ be arbitrary, and let $G$ be the Grashof number given in \eqref{Grashof_3}. Suppose that $v$ is a solution of \eqref{Leray-alpha}, then there exists a time $t_0>0$ such that for all $t\geq t_0$ we have
\begin{subequations}
\begin{align}
\norm{v(t)}_{\Lp{2}}^2 \leq 2\nu^2\lambda_1^{-1/2}G^2,
\end{align}
and
\begin{align}\label{Leray-G}
\int_t^{t+\tau} \norm{\nabla v(s)}_{\Lp{2}}^2\,ds \leq 2(1+\tau\nu\lambda_1^{1/2})\nu G^2.
\end{align}
\end{subequations}
\end{proposition}

We also recall the following bound on the solutions $v$ in the global attractor of \eqref{Leray-alpha} that was proved in \cite{FJL}. This estimate improves the estimate in \cite{CHOT} on the enstrophy $\norm{A^{1/2} v}_{\Lp{2}}^2$.

\begin{proposition} \cite{FJL} \label{unif_bounds_3D_Leray}
Suppose that $v$ is a solution in the global attractor of \eqref{Leray-alpha}, then
\begin{align}
\norm{A^{1/2} v(t)}_{\Lp{2}}^2 \leq \tilde{c} \frac{\nu^2 G^4}{\alpha^4 \lambda_1^{3/2}}, \label{Leray-enstrophy}
\end{align}
for large $t>0$, for some dimensionless constant $\tilde{c}>0$.
\end{proposition}

\bigskip

\section{Analysis of the data assimilation algorithm}
We will prove that under certain conditions on $\mu$, the approximate solution $(\vv_1,\vv_2,\vv_3)$ of the data assimilation system \eqref{Leray-alpha-DA} converges to the solution $(v_1,v_2,v_3)$ of the 3D Leray-$\alpha$ \eqref{Leray-alpha}, as $t\rightarrow \infty$, when the observables operators that satisfy \eqref{app}.

\begin{theorem}
Suppose $I_h$ satisfy \eqref{app} and $\mu>0$ and $h>0$ are chosen such that $\mu c_0^2h^2\leq \nu$, where $c_0$ is the constant in \eqref{app}.
Let $v(t, x,y,z)$ be a strong solution of the Leray-$\alpha$ model \eqref{Leray-alpha} and choose $\mu>0$ large enough such that
\begin{align}\label{Leray_alpha_mu_1}
\mu \geq 2 \frac{c\tilde{c} \nu G^4}{\alpha^4\lambda_1},
\end{align}
\and $h>0$ small enough such that $\mu c_0^2h^2\leq \nu$, where the constants $c, \tilde{c}$, and $c_0$ appear in \eqref{beta}, \eqref{Leray-enstrophy} and \eqref{app_F}, respectively.

If the initial data $\vv_0 \in V $ and $f\in H$, then the continuous data assimilation system \eqref{Leray-alpha-DA} possess a unique global strong solution $\vv(t, x,y,z)= (\vv_1(t, x,y,z), \vv_2(t , x,y,z),$ $\vv_3(t , x,y, z))$ that satisfies
\begin{align*}
\vv \in C([0,T];V ) \cap L^2([0,T];\mathcal{D}(A)), \qquad \text{and} \qquad \od{\vv}{t} \in L^2([0,T];H).
\end{align*}

Moreover, the solution $\vv(t, x,y,z)$ depends continuously on the initial data $\vv_{in}$ and it satisfies
$$\norm{v(t)-\vv(t)}_{\Lpd{2}} \rightarrow 0,$$ at exponential rate, as $t \rightarrow \infty$.
\end{theorem}

\begin{proof}
Define $\tilde{p} = p - \p$, $\dU = u-\uu$, and $\dV = v-\vv$, thus $\dV = \dU - \aa^2\Delta\dU$. Then $\dV_1$, $\dV_2$ and $\dV_3$ satisfy the equations

\begin{subequations}
\begin{align}
\label{leray-comp1}
\pd{\dV_1}{t} - \Delta\dV_1+\dU_1\pp_x v_1 + \dU_2\pp_yv_1 + \dU_3\pp_zv_1 + (\uu\cdot\nabla)\dV_1 +\pp_x \tilde{p}& = -\mu I_h(\dV_1), \\
\label{leray-comp2}
\pd{\dV_2}{t} - \Delta\dV_2+\dU_1\pp_x v_2 + \dU_2\pp_yv_2 + \dU_3\pp_zv_2 + (\uu\cdot\nabla)\dV_2 +\pp_y \tilde{p} &=-\mu I_h(\dV_2), \\
\label{leray-comp3}
\pd{\dV_3}{t} - \Delta\dV_3+\dU_1\pp_x v_3 + \dU_2\pp_yv_3 + \dU_3\pp_zv_3 + (\uu\cdot\nabla)\dV_3+ \pp_z \tilde{p}& = 0, \\
\pp_x \dV_1 + \pp_y\dV_2 + \pp_z\dV_3 =  \pp_x \dU_1 + \pp_y\dU_2 + \pp_z\dU_3&= 0 \label{div_dV}.
\end{align}
\end{subequations}

Since we assume that $v$ is a reference solution of system \eqref{Leray-alpha}, then it is enough to show the existence and uniqueness of the difference $\dV$. In the proof below, we will derive formal \textit{a-priori} bounds on $\dV$, under appropriate conditions on $\mu$ and $h$. These \textit{a-priori} estimates, together with the global existence and uniqueness of the solution $v$, form the key elements for showing the global existence of the solution $\vv$ of system \eqref{Leray-alpha-DA}. The convergence of the approximate solution $\vv$ to the exact reference solution $v$ will also be established under the tighter condition on the nudging parameter $\mu$ as stated in \eqref{Leray_alpha_mu_1}. Uniqueness can then be obtained using similar energy estimates.

The estimates we provide in this proof are formal, but can be justified by the Galerkin approximation procedure
and then passing to the limit while using the relevant compactness theorems. We
will omit the rigorous details of this standard procedure (see, e.g., \cite{Constantin_Foias_1988, Robinson, Temam_2001_Th_Num}) and provide only the formal \textit{a-priori} estimates.

Taking the $\Lp{2}$-inner product of \eqref{leray-comp1}, \eqref{leray-comp2} and  \eqref{leray-comp3} with $\dV_1$, $\dV_2$ and $\dV_3$, respectively, and obtain
\begin{align*}
\frac{1}{2}\od{}{t}\norm{\dV_1}_{\Lp{2}}^2 + \nu\norm{A^{1/2}\dV_1}_{\Lp{2}}^2 &\leq \abs{J_1} - (\pp_x\dP,\dV_1) - \mu(I_h(\dV_1) ,\dV_1),\\
\frac{1}{2}\od{}{t}\norm{\dV_2}_{\Lp{2}}^2 + \nu\norm{A^{1/2}\dV_2}_{\Lp{2}}^2 &\leq \abs{J_2} - (\pp_y\dP,\dV_2) - \mu(I_h(\dV_2) ,\dV_2),\\
\frac{1}{2}\od{}{t}\norm{\dV_3}_{\Lp{2}}^2 + \nu\norm{A^{1/2}\dV_1}_{\Lp{2}}^2 &\leq \abs{J_3} - (\pp_z\dP,\dV_3),
\end{align*}
where
\begin{align*}
J_1 := J_{1a} + J_{1b} + J_{1c} := (\dU_1\pp_x v_1, \dV_1) + (\dU_2\pp_yv_1, \dV_1) + (\dU_3\pp_zv_1, \dV_1),  \\
J_2 := J_{2a} + J_{2b} + J_{2c} := (\dU_1\pp_x v_2, \dV_2) + (\dU_2\pp_yv_2, \dV_2) + (\dU_3\pp_zv_2, \dV_2),  \\
J_3 := J_{3a} + J_{3b} + J_{3c} := (\dU_1\pp_x v_3, \dV_3) + (\dU_2\pp_yv_3, \dV_3) + (\dU_3\pp_zv_3, \dV_3).
\end{align*}

By H\"older inequality, inequality \eqref{chot} and the Sobolev inequality \eqref{sobolev}, we can show that
\begin{align}\label{11}
|J_{1a}| = \abs{(\dU_1\pp_x v_1, \dV_1)} &\leq \|\pp_x v_1\|_{\Lp{2}} \norm{\dU_1}_{\Lp{\infty}} \norm{\dV_1}_{\Lp{2}}\notag \\
&\leq c{\la_1^{-1/4}} \norm{\pp_xv_1}_{\Lp{2}} \|\dU_1\|_{\Hp{2}}\norm{\dV_1}_{\Lp{2}}\notag \\
&\leq c\la_1^{3/4} \norm{\pp_xv_1}_{\Lp{2}} \norm{\dV_1}^2_{\Lp{2}}\notag \\
&\leq c\la_1^{1/4} \norm{\pp_xv_1}_{\Lp{2}} \norm{\dV_1}_{\Lp{2}}\norm{A^{1/2}\dV_1}_{\Lp{2}}\notag\\
&\leq \frac{\nu}{8}\norm{A^{1/2}\dV_1}_{\Lp{2}}^2 + \frac{c\la_1^{1/2}}{\nu} \norm{\pp_xv_1}^2_{\Lp{2}} \norm{\dV_1}^2_{\Lp{2}}.
\end{align}
Using similar analysis as above, 
we obtain the following estimates:
\begin{align}
|J_{1b}|  = \abs{(\dU_2\pp_yv_1, \dV_1)} &\leq \frac{\nu}{8}\norm{A^{1/2}\dV_2}_{\Lp{2}}^2 + \frac{c\la_1^{1/2}}{\nu} \norm{\pp_yv_1}^2_{\Lp{2}} \norm{\dV_1}^2_{\Lp{2}},\label{22}\\
|J_{1c}| =\abs{(\dU_3\pp_zv_1, \dV_1)} &\leq \frac{\nu}{20}\norm{A^{1/2}\dV_3}_{\Lp{2}}^2 + \frac{c\la_1^{1/2}}{\nu} \norm{\pp_zv_1}^2_{\Lp{2}} \norm{\dV_1}^2_{\Lp{2}}, \label{33}\\
|J_{2a}|  = \abs{(\dU_1\pp_x v_2, \dV_2)} &\leq \frac{\nu}{8}\norm{A^{1/2}\dV_1}_{\Lp{2}}^2 + \frac{c\la_1^{1/2}}{\nu} \norm{\pp_xv_2}^2_{\Lp{2}} \norm{\dV_2}^2_{\Lp{2}},\label{44}\\
|J_{2b}| = \abs{(\dU_2\pp_yv_2, \dV_2)} &\leq \frac{\nu}{8}\norm{A^{1/2}\dV_2}_{\Lp{2}}^2 + \frac{c\la_1^{1/2}}{\nu} \norm{\pp_yv_2}^2_{\Lp{2}} \norm{\dV_2}^2_{\Lp{2}},\label{55}\\
|J_{2c}| = \abs{(\dU_3\pp_zv_2, \dV_2)}&\leq \frac{\nu}{20}\norm{A^{1/2}\dV_3}_{\Lp{2}}^2 + \frac{c\la_1^{1/2}}{\nu} \norm{\pp_zv_2}^2_{\Lp{2}} \norm{\dV_2}^2_{\Lp{2}},\label{66}\\
|J_{3a}| = \abs{(\dU_1\pp_x v_3, \dV_3)}&\leq \frac{\nu}{20}\norm{A^{1/2}\dV_3}_{\Lp{2}}^2 + \frac{c\la_1^{1/2}}{\nu} \norm{\pp_xv_3}^2_{\Lp{2}} \norm{\dV_1}^2_{\Lp{2}}\label{77},\\
|J_{3b}| =\abs{(\dU_2\pp_yv_3, \dV_3)}&\leq \frac{\nu}{20}\norm{A^{1/2}\dV_3}_{\Lp{2}}^2 + \frac{c\la_1^{1/2}}{\nu} \norm{\pp_yv_3}^2_{\Lp{2}} \norm{\dV_2}^2_{\Lp{2}}.\label{88}
\end{align}

Next, using integration by parts and the divergence free condition \eqref{div_dV} we obtain
\begin{align*}
J_{3c} = (\dU_3\pp_zv_3, \dV_3) &= -(v_3, \pp_z(\dU_3\dV_3))\notag \\
& = -(v_3,\pp_z\dU_3\dV_3) - (v_3, \dU_3\pp_z\dV_3)\notag \\
& = (v_3, (\pp_x\dU_1 + \pp_y\dU_2)\dV_3) + (v_3, \dU_3(\pp_x\dU_1 + \pp_y\dU_2) )\notag\\
& =:  J_{3d} + J_{3e}.
\end{align*}
Integration by parts once again implies
\begin{align*}
J_{3d} &= (v_3, (\pp_x\dU_1 + \pp_y\dU_2)\dV_3) \notag \\
& = - (v_3, \dU_1\pp_x\dV_3) - (v_3, \dU_2\pp_y\dV_3) - (\pp_xv_3, \dU_1\dV_3) - (\pp_yv_3, \dU_2\dV_3)\notag \\
&=: J_{3d1} + J_{3d2} + J_{3d3} + J_{3d4},
\end{align*}
and
\begin{align*}
J_{3e} & = (v_3, \dU_3(\pp_x\dU_1 + \pp_y\dU_2) )\\
& = - (v_3,\pp_x\dU_3\dV_1) - (v_3, \pp_y\dU_3\dV_2) - (\pp_xv_3, \dU_3\dV_1) - (\pp_yv_3, \dU_3\dV_2)\\
&=: J_{3e1} + J_{3e2} + J_{3e3} + J_{3e4}.
\end{align*}
Using H\"older inequality, inequality \eqref{chot} and Sobolev inequality \eqref{sobolev}, we have
\begin{align*}
|J_{3d1}|  = |(v_3, \dU_1\pp_x\dV_3)| &\leq \norm{v_3}_{\Lp{2}}\norm{\dU_1}_{\Lp{\infty}}\norm{\pp_x\dV_3}_{\Lp{2}} \notag \\
& \leq c \la_1^{-1/4} \norm{v_3}_{\Lp{2}}\norm{\dU_1}_{\Hp{2}}\norm{\pp_x\dV_3}_{\Lp{2}}\notag\\
&\leq c\la_1^{3/4}\norm{v_3}_{\Lp{2}}\norm{\dV_1}_{\Lp{2}}\norm{\pp_x\dV_3}_{\Lp{2}}\notag\\
&\leq c\la_1^{1/4}\norm{A^{1/2} v_3}_{\Lp{2}}\norm{\dV_1}_{\Lp{2}}\norm{\pp_x\dV_3}_{\Lp{2}}\notag \\
&\leq\frac{\nu}{20}\norm{\pp_x\dV_3}^2_{\Lp{2}} + \frac{c\la_1^{1/2}}{\nu} \norm{A^{1/2} v_3}^2_{\Lp{2}}\norm{\dV_1}^2_{\Lp{2}},
\end{align*}
and similarly,
\begin{align*}
|J_{3d2}| = |(v_3, \dU_2\pp_y\dV_3)| &\leq\frac{\nu}{20}\norm{\pp_y\dV_3}^2_{\Lp{2}} + \frac{c\la_1^{1/2}}{\nu} \norm{A^{1/2} v_3}^2_{\Lp{2}}\norm{\dV_2}^2_{\Lp{2}}.
\end{align*}
By a similar argument as in \eqref{11}, we can show that
\begin{align*}
|J_{3d3}| = \abs{(\pp_xv_3, \dU_1\dV_3)}
&\leq\frac{\nu}{20}\norm{A^{1/2}\dV_3}^2_{\Lp{2}} + \frac{c\la_1^{1/2}}{\nu} \norm{\pp_xv_3}^2_{\Lp{2}}\norm{\dV_1}^2_{\Lp{2}},
\end{align*}
and
\begin{align*}
|J_{3d4}| =\abs{(\pp_yv_3, \dU_3\dV_2)}
&\leq\frac{\nu}{20}\norm{A^{1/2}\dV_3}^2_{\Lp{2}} + \frac{c\la_1^{1/2}}{\nu} \norm{\pp_yv_3}^2_{\Lp{2}}\norm{\dV_2}^2_{\Lp{2}}.
\end{align*}
Thus,
\begin{align*}
|J_{3d}| \leq \frac{\nu}{5}\norm{A^{1/2}\dV_3}^2_{\Lp{2}} + \frac{c\la_1^{1/2}}{\nu}\norm{A^{1/2} v_3}^2_{\Lp{2}}\left( \norm{\dV_1}^2_{\Lp{2}} + \norm{\dV_2}^2_{\Lp{2}}     \right).
\end{align*}
We apply similar calculations to $J_{3e}$ and obtain

\begin{align*}
|J_{3e}|  \leq \frac{\nu}{5}\norm{A^{1/2}\dV_3}^2_{\Lp{2}} + \frac{c\la_1^{1/2}}{\nu}\norm{A^{1/2} v_3}^2_{\Lp{2}}\left( \norm{\dV_1}^2_{\Lp{2}} + \norm{\dV_2}^2_{\Lp{2}} \right).
\end{align*}
This yield
\begin{align}\label{99}
\abs{J_{3c}} &= \abs{(\dU_3\pp_zv_3, \dV_3)}\notag \\
& \leq \frac{2\nu}{5}\norm{A^{1/2}\dV_3}^2_{\Lp{2}} + \frac{c\la_1^{1/2}}{\nu}\norm{A^{1/2} v_3}^2_{\Lp{2}}\left( \norm{\dV_1}^2_{\Lp{2}} + \norm{\dV_2}^2_{\Lp{2}} \right).
\end{align}
Young's inequality and the assumption $\mu c_0^2h^2 \leq \nu$ imply that
\begin{align}\label{1010}
-\mu(I_h(\dV_i),\dV_i)
&= -\mu (I_h(\dV_i)-\dV_i,\dV_i) - \mu\norm{\dV_i}_{\Lp{2}}^2 \notag \\
& \leq \mu c_0 h \norm{\dV_i}_{\Lp{2}}\norm{A^{1/2}\dV_i}_{\Lp{2}} - \mu \norm{\dV_i}_{\Lp{2}}^2\notag \\
& \leq \frac{\nu}{2}\norm{A^{1/2}\dV_i}_{\Lp{2}}^2  - \frac{\mu}{2} \norm{\dV_i}_{\Lp{2}}^2, \qquad i=1,2.
\end{align}
Also we note that
\begin{align}\label{1111}
(\partial_x\dP, \dV_1) + (\partial_y\dP, \dV_2) + (\partial_z\dP, \dV_3) =0,
\end{align}
due to integration by parts, the boundary conditions, and the divergence free condition \eqref{div_dV}.

Combining all the bounds \eqref{11}--\eqref{1111} and denoting $\norm{\dV_H}^2_{\Lp{2}}: = \norm{\dV_1}^2_{\Lp{2}}+ \norm{\dV_2}^2_{\Lp{2}}$,  we obtain
\begin{align*}
\od{}{t} \norm{\dV}_{\Lp{2}}^2 + \frac\nu2\norm{A^{1/2}\dV}_{\Lp{2}}^2 \leq \left(\frac{c\la_1^{1/2}}{\nu}\norm{A^{1/2} v}^2_{\Lp{2}} - \mu\right)\norm{\dV_H}^2_{\Lp{2}},
\end{align*}
or, using Poincar\'e inequality \eqref{poincare}, we have
\begin{align}
\od{}{t} \norm{\dV}_{\Lp{2}}^2 + \frac{\nu\la_1}{2}\norm{\dV}_{\Lp{2}}^2 + \beta(t)\norm{\dV_H}^2_{\Lp{2}}\leq 0,
\end{align}
where
\begin{align}\label{beta}
\beta(t) := \mu - \frac{c\la_1^{1/2}}{\nu}\norm{A^{1/2} v}^2_{\Lp{2}}.
\end{align}

Since $v(t)$ is a solution in the global attractor of \eqref{Leray-alpha}, then $\norm{A^{1/2} v}_{\Lp{2}}^2$ satisfies the bound \eqref{Leray-enstrophy} for $t>t_0$, for some large enough $t_0>0$. Now, the assumption \eqref{Leray_alpha_mu_1}, yields
\begin{align*}
\od{}{t} \norm{\dV}_{\Lp{2}}^2 +\min\bigg\{\frac{\nu\la_1}{2}, \,\frac{\mu}{2}\bigg\}\norm{\dV}^2_{\Lp{2}}\leq 0,
\end{align*}
for $t>t_0$. By Gronwall's lemma, we obtain
$$\norm{v(t)-\vv(t)}_{\Lp{2}}^2 = \norm{\dV(t)}_{\Lp{2}}\rightarrow 0,$$ at an exponential rate, as $t\rightarrow\infty$.
\end{proof}

\bigskip
\section*{Acknowledgements}
We would like to thank Professor M.\ Ghil for the stimulating exchange concerning the Charney conjecture and for pointing out to us some of the relevant references.  E.S.T.\ is thankful to the kind hospitality of ICERM, Brown University, where part of this work was completed. The work of A.F.\ is supported in part by the NSF grant  DMS-1418911. The work of E.L.\ is supported in part by the ONR grant N0001416WX01475, N0001416WX00796 and HPC grant W81EF61205768.  The work of  E.S.T.\  is supported in part by the ONR grant N00014-15-1-2333 and the NSF grants DMS-1109640 and DMS-1109645.

\bigskip

\end{document}